\documentclass{amsart}
\usepackage[english]{babel}
\usepackage[usenames,dvipsnames]{pstricks}
\usepackage{epsfig}
\usepackage{amssymb,amsmath}
\usepackage[all]{xy}

\newtheorem{theorem}{Theorem}[section]
\newtheorem{proposition}[theorem]{Proposition}

\newtheorem{lemma}[theorem]{Lemma}
\newtheorem{example}[theorem]{Example}

\newcommand{\cs}{\mathcal{C}}               
\newcommand{\pcs}{\mathcal{C}^I}  
\newcommand{\ws}{\mathcal{W}}               

\newcommand{\TC}{\mathrm{TC}}  
\newcommand{\cat}{\mathrm{cat}} 
\newcommand{\g}{\mathsf{g}}

\newcommand{\Ker}{\mathord{\mathrm{Ker}}}

\newcommand{\Id}{\mathord{\mathrm{Id}}}

\newcommand{\csec}{\mathord{\mathrm{csec}}}

\newcommand{\nil}{\mathord{\mathrm{nil}}}

\newcommand{\RR}{\mathord{\mathbb{R}}}
\newcommand{\ZZ}{\mathord{\mathbb{Z}}}

\title{A Topologist's View of Kinematic Maps and Manipulation Complexity}
\author[Petar Pave\v si\'c]{Petar Pave\v si\'c}
\address{Faculty of Mathematics and Physics, University of Ljubljana, Ljubljana, Slovenia}
\email{\rm{petar.pavesic@fmf.uni-lj.si}}
\thanks{The author was partially supported by the Slovenian Research Agency grant P1-02920101.}
\subjclass[2010]{}
\begin{document}
\begin{abstract}
In this paper we combine a survey of the most important topological properties of kinematic maps that appear in robotics,
with the exposition of some basic results regarding the topological complexity of a map. In particular, we discuss 
mechanical devices that consist of rigid parts connected by joints and show how the geometry of the joints determines 
the forward kinematic map that relates the configuration of joints with the pose of the end-effector of the device. 
We explain how to compute the dimension of the joint space and describe topological obstructions for a kinematic
map to be a fibration or to admit a continuous section. In the second part of the paper we define the complexity 
of a continuous map
and show how the concept can be viewed as a measure of the difficulty to find a robust manipulation plan for 
a given mechanical 
 device. We also derive some basic estimates for the complexity and relate it to the degree of instability of a manipulation
plan.
\end{abstract}
\maketitle

\section{Introduction}
\label{sec:Introduction}

\emph{Motion planning} is one of the basic tasks in robotics: it requires finding a suitable continuous motion 
that transforms (or moves) a robotic device from a given initial position to a desired final position. A motion plan 
is usually required to be \emph{robust}, i.e. the path of the robot must be a continuous function of the input data 
given by the initial and final position of the robot. 
Michael Farber \cite{Farber:TCMP} introduced the concept of \emph{topological complexity} as a   
measure of the difficulty finding continuous motion plans that yield robot trajectories for every admissible pair 
of initial and final points (see also \cite{Farber:ITR}). 
However, finding a suitable trajectory for a robot is only part of the problem, because
robots are complex devices and one has to decide how to move each  robot component so that the entire device moves along 
the required  trajectory. To tackle this more general
\emph{robot manipulation problem} one must take into account the relation between the internal states of robot parts 
(that form the so-called \emph{configuration space} of the robot) and the actual pose of the robot within its
\emph{working space}. The relation between the internal states and the poses of the robot is given by the \emph{forward 
kinematic map}. In \cite{Pav:CFKM} we defined the \emph{complexity of the forward kinematic map} as a measure of 
the difficulty to find robust manipulation plans for a given robotic device.

We begin the paper with a survey of basic concepts of robotics which is intended as a motivation and background information
for various problems that appear in the study of topological complexity of configuration spaces and kinematic maps. 
Section \ref{sec:Robot kinematics} contains a brief exposition of  standard mathematical topics in robotics: description 
of the position and
orientation of rigid bodies, classification of joints that are used to connect mechanism
parts, definition of configuration and working spaces and determination of the mechanism kinematics based on the motion 
of the joints. Our exposition is by no means complete, as we limit our attention to concepts that appear in 
 geometrical and topological questions. More technical details can be found in standard books on robotics, like 
\cite{RS}, \cite{MIRM} or \cite{Waldron-Schmiedeler}.
In Section \ref{sec:Topological properties} we consider the  properties of kinematic maps that are relevant to the study 
of topological 
complexity (cf. \cite{Pav:CFKM}). In particular, we determine the dimension of the configuration space, 
discuss when a kinematic map admits a continuous section (i.e. inverse kinematic map) and when a given kinematic map 
is a fibration. We also mention some questions that arise in the kinematics of redundant manipulators. 
The results presented in Section \ref{sec:Topological properties} 
are not our original contribution, but we have made an effort to
give a unified exposition of relevant facts scattered in the literature, and to help 
a topologically minded reader to familiarize herself or himself with aspects of robotics that have motivated 
some recent work on topological complexity. 

In the second part of the paper we recall some basic facts about topological complexity 
in Section \ref{sec:TC overview} and then in Section \ref{sec:Complexity of a map} we 
introduce the relative version, the complexity of a continuous map. We discuss some subtleties in the definition 
of complexity, derive one basic estimate and present several possible applications. In Section 
\ref{sec:Instability of robot manipulation} we relate the instabilities (i.e. discontinuities) that appear in 
the manipulation of a mechanical system with the complexity of its kinematic map.

\section{Robot kinematics}
\label{sec:Robot kinematics}

In order to keep the discussion reasonably simple, we will restrict our attention to mechanical aspects of robotic
kinematics and disregard questions 
of adaptivity and communication with humans or other robots. We will thus view a robot as a mechanical
device with rigid parts connected by joints. Furthermore, we will not take into account forces or 
torques as these concepts properly belong to robot \emph{dynamics}.

A robot device consists of rigid components connected by joints that allow its parts to change their relative positions. 
To give a mathematical model of robot motion
we need to describe the position and orientation of individual parts, determine the motion restrictions caused 
by various types of joints, and compute the functional relation between the states of individual joints 
and the position and orientation of the end-effector. 

\subsection{Pose representations}

The spatial description of each part of a robot, viewed as a rigid body,  is given by its \emph{position} and \emph{orientation},
which are collectively called \emph{pose}. The position is usually given by specifying a point in $\RR^3$ 
occupied by some reference point in the robot part, and the orientation is given by an element of $SO(3)$. 
Therefore, as $\RR^3\times SO(3)$ is 6-dimensional, we need at least six coordinates to precisely locate each robot 
component in Euclidean space. The representation
of the position is usually straightforward in terms of cartesian or cylindrical coordinates, but the explicit 
description of the orientation is more complicated.
Of course, we may specify an element in $SO(3)$ by a $3\times 3$ matrix, but that requires a list of 9 coefficients
that are subject to 6 relations. Actually, 3 equations are redundant because the matrix is symmetric, 
while the remaining 3 are quadratic and involve 
all coefficients. This considerably complicates computations involving relative positions of various parts,
so most robotics courses begin with a lengthy discussions of alternative representations of rotations.
This includes description of elements of $SO(3)$ as compositions of three rotations around coordinate axes
(\emph{fixed angles} representation), or by rotations around changing axes (\emph{Euler angles} representation),
or by specifying the axes and the angle of the rotation (\emph{angle-axis} representation). While these representations
are more efficient in terms of data needed to specify a rotation, explicit formulas always have singularities,
where certain coefficients are undefined. This is hardly surprising, as we know that $SO(3)$ cannot be 
parametrized by a single 3-dimensional chart. Other explicit descriptions of elements of $SO(3)$ include those 
by quaternions and by
matrix exponential form. See \cite{RS}, \cite{Waldron-Schmiedeler} for more details and transition formulas between 
different representations of spatial orientation.

The pose of a rigid body corresponds to an element of the special Euclidean group $SE(3)$, which can 
be identified with the semi-direct product of $\RR^3$ with $SO(3)$. Its elements admit a \emph{homogeneous representation} 
by $4\times 4$-matrices of the form 
$$\left(\begin{array}{cc}
R & \vec t\\
\vec 0^T & 1 \end{array}\right)$$
where $R$ is a $3\times 3$ special orthogonal matrix representing a rotation and $\vec t$ is a 3-dimensional 
vector representing a translation.
The main advantage of this representation is that composition of motions is given by the multiplication of corresponding 
matrices. Another frequently used representation of the rigid body motion is by \emph{screw transformations}. It is based 
on the Chasles theorem which states that the motion given by a rotation around an axis passing through the center 
of mass, followed by 
a translation can be obtained by a screw-like motion given by simultaneous rotation and translation along a common axis 
(parallel to the previous). See \cite[section 1.2]{Waldron-Schmiedeler} 
for more details about joint position and orientation representations.

Explicit representations of rigid body pose may be complicated but are clearly unavoidable when it 
comes to numerical computations.
Luckily, topological considerations mostly rely on geometric arguments and rarely involve explicit formulae.

\subsection{Joints and mechanisms}

Rigid components of a robot mechanism are connected by \emph{joints}, i.e. parts of the components' surfaces that are in 
contact with each other. The geometry of the contact surface restricts the relative motion of the components connected
by a joint. Although most robot mechanisms employ only two basic kinds of joints, we will briefly describe 
a general classifications of various joint types. First of all, two objects can be in contact along a proper
surface (these are called \emph{lower pair} joint), along a line, or even along isolated points (in the case of 
\emph{upper pair} joints). 
There are six basic types of lower pair joints, the most important being the first three: 

\emph{Revolute joints:} the contact between the bodies is along a surface of revolution, which allows rotational motion.
Revolute joints are usually abbreviated by (R) and the corresponding motion has one degree of freedom (1 DOF).

\emph{Prismatic joints:}the bodies are in contact along a prismatic surface. Prismatic joints are abbreviated as (P) and 
admit a rectilinear motion with one degree of freedom.

\emph{Helical joints:} the bodies are in contact along a helical surface. Helical joints are abbreviated as (H) and allow screw-like 
motion with one degree of freedom. 

\begin{figure}[ht]
    \centering
    \includegraphics[scale=0.5]{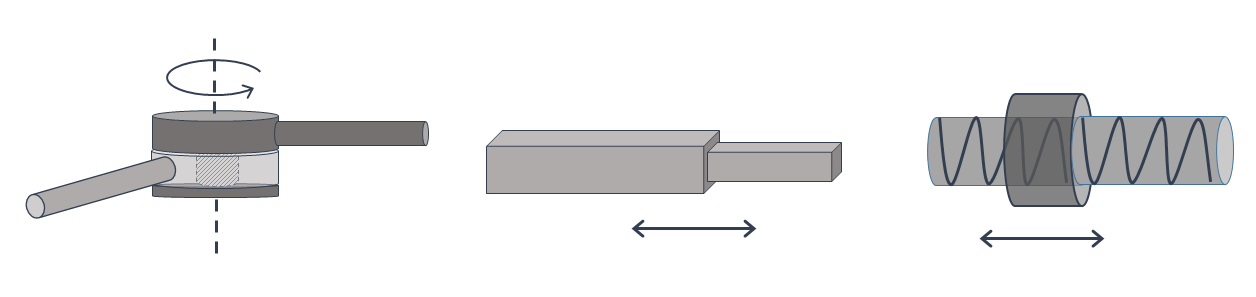}
    \caption{Revolute (R), prismatic (P) and helical (H) joints.}
    \label{fig: RPH joints}
\end{figure} 

\emph{Cylindrical joints:}, denoted (C), where the bodies are in contact along a cylindrical surface. They allow simultaneous sliding and 
rotation, so the corresponding motion has two degrees of freedom.

\emph{Spherical joints:}, denoted (S), with bodies in contact along a spherical surface and allowing motion with three degrees of freedom.

\emph{Planar joints:}, denoted (E) with contact along a plane with three degrees of freedom (plane sliding and rotation).

\begin{figure}[ht]
    \centering
    \includegraphics[scale=0.5]{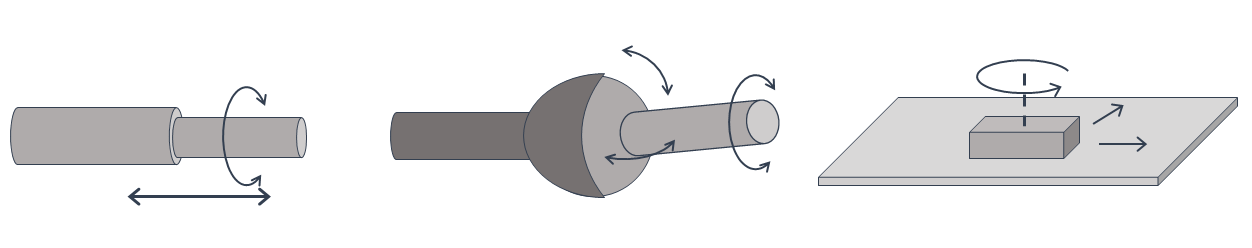}
    \caption{Cylindrical (C), spherical (S) and planar (E) joints.}
    \label{fig: CSP joints}
\end{figure} 

While the revolute, prismatic and helical joints can be easily actuated by motors or pneumatic cylinders, this is not 
the case for the remaining three types, because they have two or three degrees of freedom and each degree of freedom 
must be separately 
actuated. As a consequence, they are used less frequently in robotic mechanisms and almost exclusively as passive 
joints that restrict the motion of the mechanism. 

Higher pair joints are also called rolling joints, being characterized by a one-dimensional contact between the bodies, 
like a cylinder rolling on a plane, or by zero-dimensional contact, like a sphere rolling on a surface. They too appear
only as passive joints.

A complex of rigid bodies, connected by joints which as a whole allow at least one degree of freedom, forms 
a \emph{mechanism} (if no movement is possible, it is called a \emph{structure}). 
A mechanism is often schematically described by a graph whose vertices
are the joints and edges correspond to the components. The graph may be occasionally complemented with symbols 
indicating the type of each joint 
or its degree of freedom. A manipulator whose graph is a path is called \emph{serial chain}. 
This class is sometimes extended to
include manipulators with tree-like graphs as in robot hands with fingers or in some gripping mechanisms. A serial 
manipulator necessarily  contains only actuated joints and is often codified by listing the symbols for its joints. 
For example (RPR) denotes a chain in which the first joint is revolute, second is prismatic and the third is 
again revolute. Typical serial chains are various kinds of robot arms.

Manipulators whose graphs contain one or more cycles are \emph{parallel}. Typical parallel mechanisms are various
lifting platforms (see Figure \ref{fig: ser-par}). We will see later that the kinematics of serial mechanisms is quite 
different from the kinematics of parallel mechanisms and requires different methods of analysis. 

\begin{figure}[ht]
    \centering
    \includegraphics[scale=0.5]{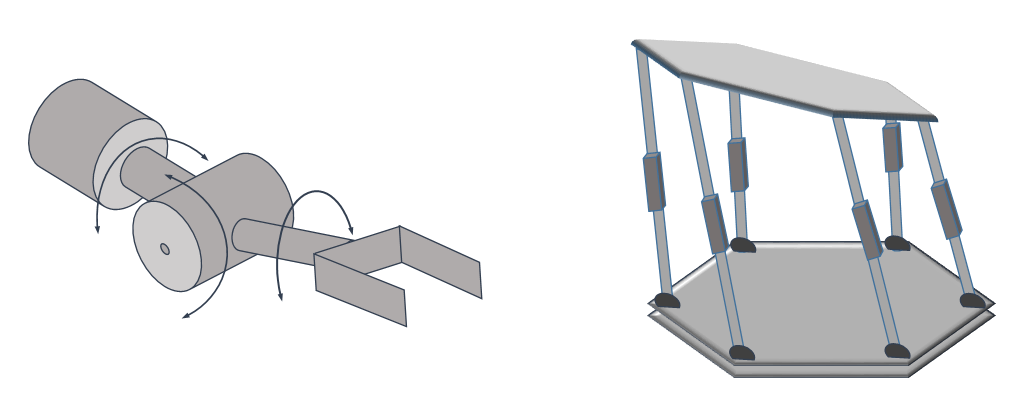}
    \caption{Serial (RRR) robot 'triple-roll wrist' and parallel Stewart platform mechanism where each leg is of type (SPS).}
    \label{fig: ser-par}
\end{figure}

\subsection{Kinematic maps}

Let us consider a simple example -- a pointing mechanism with two revolute joints as in Figure \ref{fig: pointing}. 
Since each of the joints can rotate a full circle, we can specify their position by giving two angles, or  better, two  
points on the unit circle $S^1$. Every choice of angles uniquely determines the longitude and the latitude of a point on 
the sphere. Thus we obtain an elementary example of a kinematic mapping $f\colon S^1\times S^1\to S^2$, explicitly given
in terms of geographical coordinates as $f(\alpha,\beta)=(\cos\alpha\cos\beta,\cos\alpha\sin\beta,\sin\alpha)$,
so that $\alpha$ is the latitude and $\beta$ is the longitude of a point on the sphere.

\begin{figure}[ht]
    \centering
    \includegraphics[scale=0.5]{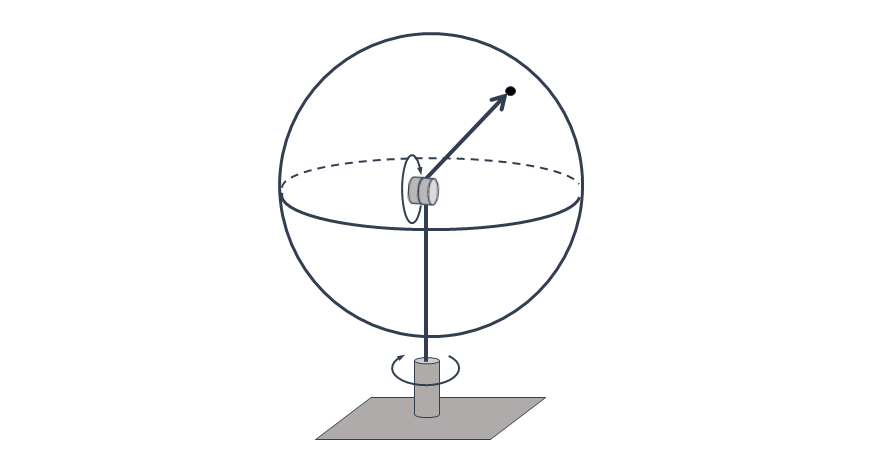}
    \caption{Pointing (RR) mechanism.}
    \label{fig: pointing}
\end{figure}

Given two bodies connected by a joint $J$, we define the \emph{joint space} of $J$ as the subspace (usually a submanifold)
of $SE(3)$ that correspond to all possible relative displacements of the two bodies. So the joint space of a revolute joint
is (homeomorphic to) $S^1$, joint spaces of prismatic and helical joints are closed segments $B^1$, joint space of 
a cylindrical
joint is $B^1\times S^1$, and the joint spaces of spherical and planar joints are $B^2\times S^1$ (note that theoretically 
a spherical joint should have the entire $SO(3)$ as a joint space, but such a level of mobility cannot be technically achieved). 

The \emph{Joint space} of a robot manipulator $\mathcal{M}$ is the product of the joint spaces of its joints. 
Its \emph{configuration space} $\cs(\mathcal{M})$ is the subset  of the joint space of $\mathcal{M}$, consisting of 
values for the joint variables that satisfy all constraints determined by a geometrical realization of the manipulator. 

The component of a manipulator  that performs a desired task is called an \emph{end-effector}. 
The \emph{kinematic mapping} for $\mathcal{M}$ is the function $f\colon \cs(\mathcal{M})\to SE(3)$ that to every admissible
configuration of joints assigns the pose of the end-effector. The image of the kinematic mapping is called 
the \emph{working space}
of $\mathcal{M}$ and is denoted $\ws(\mathcal{M})$. Often we only care about the position (or orientation) of 
the end-effector
and thus consider just the projections of the working space to $\RR^3$ (or $SO(3)$). The \emph{inverse kinematic mapping}
for $\mathcal{M}$ is a right inverse (section) for $f$, i.e. a function $s\colon\ws(\mathcal{M})\to\cs(\mathcal{M})$,
satisfying $f\circ s=\Id_{\ws(\mathcal{M})}$. We will see later that many kinematic maps (especially for serial chains) 
do not admit continuous inverse kinematic maps. 

In order to study or manipulate a robot mechanism one must explicitly compute its forward kinematic map. 
To this end it is necessary to describe 
the poses of the joints, and this cannot be done in absolute terms because the movement of each joint can 
change the pose of other joints. For a serial chain it makes sense to specify the pose of each joint relatively to 
the previous joint in
the chain. In other words, we can fix a reference frame for each joint in the chain, and then form a list starting with 
the pose of 
the first joint, followed by the difference between the poses of the first joint and the second joint, followed by 
the difference between 
the poses of the second joint and the third joint, and so on. Of course, each difference is an element of $SE(3)$, hence 
it is determined
by six parameters. However, by a judicious choice of reference frames one can reduce this to just four parameters for 
each difference, with an additional bonus of a very simple computation of the joint poses. This method, introduced 
by Denavit and Hartenberg \cite{D-H}, is the most widely used approach for the computation of the forward kinematic map 
for serial chains, so  deserves to be described in some detail. 

Assume we are given a serial chain with links and revolute joints as in Figures \ref{fig: SCARA} and 
\ref{fig: serial}. For the first joint we 
let the $z$-axis be the axis of rotation of the joint, and choose the $x$-axis and $y$-axis so to form 
a right-handed orthogonal frame. We may fix the frame
origin within the joint but that is not required, and indeed the other frame origins usually do not coincide 
with the positions of the respective joints. 
Given the frame $\mathbf{x},\mathbf{y},\mathbf{z}$ for the $i$-th joint, the frame $\mathbf{x'},\mathbf{y'},\mathbf{z'}$
 for the $(i+1)$-st joint is chosen as follows:
\begin{figure}[ht]
    \centering
    \includegraphics[scale=0.5]{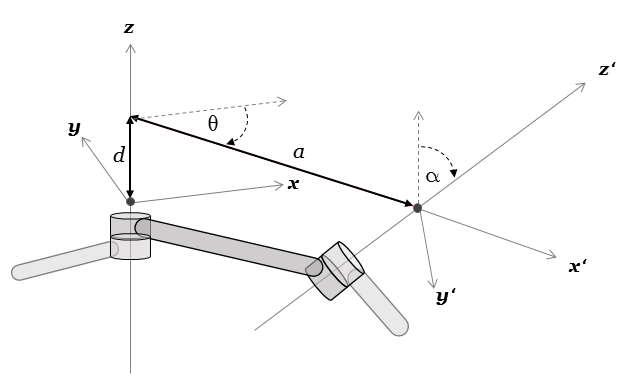}
    \caption{Denavit-Hartenberg parameters and frame convention.}
    \label{fig: DH}
\end{figure}
\begin{itemize}
\item $\mathbf{z'}$ is the axis of rotation of the joint;
\item $\mathbf{x'}$ is the line containing the shortest segment connecting $\mathbf{z}$ and $\mathbf{z'}$ (its
direction is thus $\mathbf{z'}\times \mathbf{z}$); 
if $\mathbf{z}$ and $\mathbf{z'}$ are parallel, we usually
choose the line through the origin of the $i$-th frame:
\item $\mathbf{y'}$ forms a right-handed orthogonal frame with $\mathbf{x'}$ and $\mathbf{z'}$.
\end{itemize}
The relative position of the frame $\mathbf{x'},\mathbf{y'},\mathbf{z'}$ with respect to the frame 
$\mathbf{x},\mathbf{y},\mathbf{z}$ is given by four Denavit-Hartenberg 
parameters $d,\theta,a,\alpha$ (see Figure \ref{fig: DH}), where 
\begin{itemize}
\item $d$ and $\theta$ are the distance and the angle between the axes $\mathbf{x}$ and $\mathbf{x'}$ 
\item $a$ and $\alpha$ are the distance and angle between the axes $\mathbf{z}$ and $\mathbf{z'}$. 
\end{itemize}
Using the above procedure, 
one can describe the structure of a serial chain with 
$n$ joints by giving the initial frame and a list of $n$ quadruplets of Denavit-Hartenberg parameters. 
Moreover, Denavit-Hartenberg approach can be easily extended to 
handle combinations of prismatic and revolute joints, and to take into account some exceptional configurations, 
e.g. coinciding axes of rotation. 

Once the structure of the serial chain is coded in terms of Denavit-Hartenberg parameters it is not difficult 
to write explicitly the corresponding kinematic map as a product 
of rotation and translation matrices. It is important to note that, for each joint $\theta$ is precisely the joint
parameter describing the joint rotation by the angle $\theta$. 
We omit the tedious computation and just mention that the kinematic map of a robot arm with $n$ 
revolute joints is a product of 
$n$ Denavit-Hartenberg matrices of the form
$$\left(\begin{array}{cccc} 
\cos\theta  &  -\cos\alpha\sin\theta & \sin\alpha \sin\theta & a\cos\theta \\
\sin\theta  &  \cos\alpha\cos\theta & -\sin\alpha \cos\theta & a\sin\theta \\  
         0   &  \sin\alpha             & \cos\alpha             & d \\
         0   &       0                 &               0        &  1
\end{array}\right)$$
where the parameters range over all joints in the chain.

\section{Topological properties of kinematic maps}
\label{sec:Topological properties}

In this section we present a few results of topological nature concerning configuration spaces, working spaces and kinematic
 maps. Although the complexity can be defined for any continuous map,
these results hint at additional conditions that can be reasonably imposed in concrete applications, 
and conversely, show that some common topological assumptions 
(e.g. that the kinematic is a fibration) can be too much to ask for in practical applications.

\subsection{Mobility} 

In order to form a \emph{mechanism} a set of bars and joints must be mobile, otherwise it is more properly called 
a \emph{structure}. The \emph{mobility} of a robot mechanism is usually defined to be the number of its degrees of 
freedom. In more mathematical terms, we can identify mobility as the dimension of the configuration space (at least 
when $\cs$ is 
a manifold or an algebraic variety). The mobility of a serial mechanism is easily determined: it is the sum of the 
degrees of freedom of its joints (which usually coincides with the number of joints,
because actuated joints have one degree of freedom). 
In parallel mechanisms links that form cycles reduce the mobility of the mechanism. Assume that a mechanism consists 
of $n$ moving bodies that are connected
directly or indirectly to a fixed frame. If they are allowed to move independently in the space then 
the configurations space is $6n$-dimensional. Each joint introduces some constraints and generically 
reduces the dimension of the configuration space by $6-f$, where $f$ is the degree of freedom of the joint. Therefore, 
if there are $g$ joints whose degrees of freedom are $f_1,f_2,\ldots,f_g$, and if they
are independent, then the mobility of the system is 
$$M=6n-\sum_{i=1}^g (6-f_i)= 6(n-g)+\sum_{i=1}^g f_i.$$
This is the so called \emph{Gr\"ubler formula} (sometimes called Chebishev-Gr\"ubler-Kutzbach formula).
If the mechanism is planar, then the mobility of each body is 3-dimensional (two planar coordinates and the plane rotation), and each joint introduces $3-f$ constraints. The corresponding planar Gr\"ubler 
formula gives the mobility as 
$$M= 3(n-g)+\sum_{i=1}^g f_i.$$
For example, in a simple planar linkage with four links (of which one is fixed), connected by four revolute joints the mobility is $M=3\times(3-4)+4=1$. 
\begin{figure}[ht]
    \centering
    \includegraphics[scale=0.5]{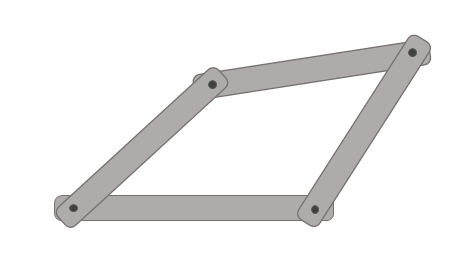}
    \caption{Four-bar linkage with mobility 1.}
    \label{fig: four bars}
\end{figure}
Observe however that the Gr\"ubler formula relies on the assumption that the constraints are independent. 
In more complicated mechanisms relations between motions of adjacent joints may lead to redundant degrees of freedom 
in the sense that some motions are always related. For example, in a (SPS) configuration, with a prismatic joint between
two spherical joints, a rotation of one spherical joint is transmitted to an equivalent 
rotation of the other spherical joint. Thus the resulting degree of freedom is not $7=3+1+3$ but $6$. 
For example, in the Stewart platform shown in Figure \ref{fig: ser-par} there is the fixed base, 13 mobile links, and 
18 joints, of which 6 in the struts are prismatic with one degree  of freedom,
and the remaining 12 at both platforms have three degrees of freedom each. Thus by the Gr\"ubler formula, the mobility of
the Stewart platform should be equal to
$M=(13-18)\times 6+12\times 3+6\times 1=12$, but in fact each leg has one redundant degree of freedom, so the mobility 
of the Stewart platform is 6. Observe that to achieve all positions in the configuration
space, one must actuate at least $M$ joints (assuming that only joints with one degree of freedom are actuated). In fact, 
in the Stewart platform the six prismatic joints are actuated, while the spherical joints are passive.

\subsection{Inverse kinematics}

A crucial step in robot manipulation is the determination of a configuration of joints that realizes a given pose in 
the working space. In other words, we need to find an inverse kinematic map
in order to reduce the manipulation problem to a navigation problem within $\cs$.  However, very often
we must rely on partial inverses because there are topological obstructions for the existence of an inverse kinematic map
defined on the entire working space $\ws$. The following result is due to Gottlieb \cite{Gottlieb:RFB}:

\begin{theorem}
\label{sections}
A continuous map $f\colon (S^1)^n\to \ws$ where $\ws=S^2, SO(3)$ or $SE(3)$ does not admit a continuous section.
\end{theorem}
\begin{proof}
A continuous map $s\colon \ws\to (S^1)^n$, such that $f\circ s=\Id_\ws$ induces a homomorphism between fundamental 
groups $s_\sharp\colon\pi_1(\ws)\to\pi_1(\cs)$ satisfying $f_\sharp\circ s_\sharp=\Id$. However, the identity on
torsion groups 
for $\pi_1(SO(3))=\pi_1(SE(3))=\ZZ_2$ cannot factor through the free abelian group $\pi_1((S^1)^n)=\ZZ^n$. 

Similarly, by applying the second homotopy group functor, we conclude that the identity on $\pi_2(S^2)=\ZZ$ cannot 
factor through $\pi_2((S^1)^n)=0$.
\end{proof}

As a consequence, if one wants to use a serial manipulator to move the end-effector in a spherical space around the 
device, or to control a robot arm that is able to 
assume any orientation, then the computation of joint configurations that yield a desired position or orientation 
requires a partitioning of the working space into subspaces
that admit inverse kinematics. This explains the popularity of certain robot configurations that avoid this problem. 
A typical example is the SCARA ('Selective Compliance Assembly Robot Arm') design. 
\begin{figure}[ht]
    \centering
    \includegraphics[scale=0.5]{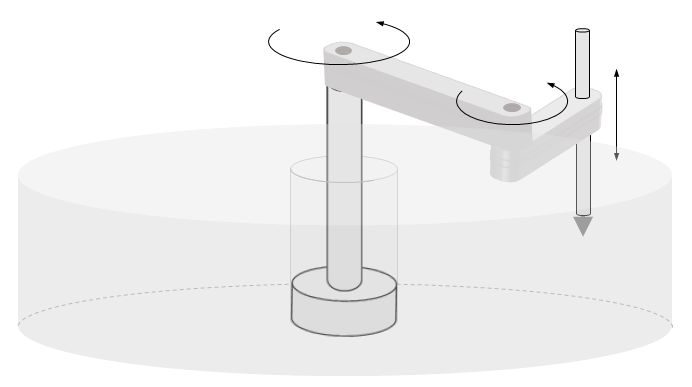}
    \caption{SCARA working space.}
    \label{fig: SCARA}
\end{figure}
Its working space is a doughnut-shaped region homeomorphic to $S^1\times I\times I$, so 
the previous theorem does not apply. Indeed, it is not difficult to obtain an inverse kinematic map for 
the SCARA robot arm.

A similar question arises if one attempts to write an explicit formula to compute
the axes of rotations in $\RR^3$.
It is well-known that for every non-identity rotation of $\RR^3$ there is a uniquely defined axis of rotation, viewed 
as an element of $\RR P^2$. For programming 
purposes it would be useful to have an explicit formula that to a matrix  $A\in SO(3)-I$  assigns a vector 
$(a,b,c)\in\RR^3$ determining the axis of the rotation represented by $A$. 
But that would amount to a factorization of the axis map $SO(3)-I\to\RR P^2$ 
through some continuous map 
$f\colon  SO(3)-I\to\RR^3-\mathbf{0}$, which cannot be done, because the axis map induces an isomorphism on 
$\pi_1(SO(3)-I)=\pi_1(\RR P^2)=\ZZ_2$, while $\pi_1(\RR^3-\mathbf{0})=0$.

\subsection{Singularities of kinematic maps}

In robotics, the term \emph{kinematic singularity} is used to denote the reduction in freedom of movement in 
the working space that arises in certain joint configurations. 
Consider for example the pointing mechanism in Figure \ref{fig: pointing} and imagine that it is steered so as to point
 toward some flying object.
If the object heads directly toward the north pole and from that point moves sidewise, then the mechanism will 
not be able to follow it in a continuous 
manner, because that would require an instantaneous rotation around the vertical axis. Similarly, if the object
flies very close to the axis through the north pole, then a continuous tracking is theoretically possible, but it may
 require infeasibly high rotational speeds. 
Both problems are caused by the fact that the poles are singular values of the forward kinematic map. More precisely,  
let us assume that
$\cs$ and $\ws$ are smooth manifolds, and $f\colon\cs\to\ws$ is a smooth map. Then $f$ induces the derivative map $f_*$ 
from the tangent bundle of $\cs$ to the tangent 
bundle of $\ws$. If $f_*$ is not onto at some point $c\in\cs$ (or equivalently, if the Jacobian of $f$ does not have maximal
 rank at $c$), then it is not possible 
to move 'infinitesimally' in certain directions from $f(c)$ while staying in a neighbourhood of $c$. 

This phenomenon is clearly visible
 in Figure \ref{fig: singularity}, which 
depicts the kinematic map of the pointing mechanism. 
\begin{figure}[ht]
    \centering
    \includegraphics[scale=0.5]{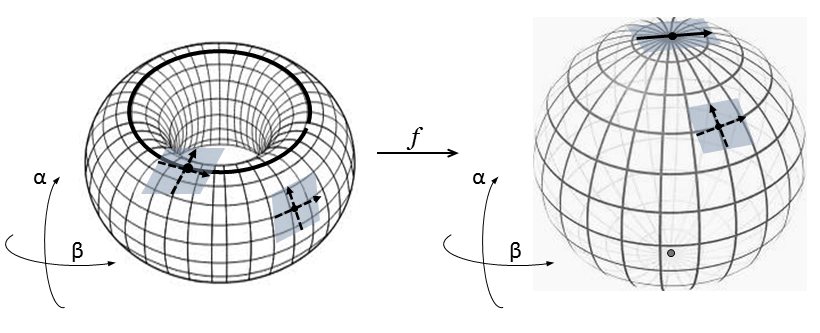}
    \caption{Singularities of the pointing mechanism.}
    \label{fig: singularity}
\end{figure}
For generic points the Jacobian of $f$ is a non-singular $2\times 2$-matrix, which means that the mechanism can move 
in any direction. However, for $\alpha=\pi/2$ the range of the Jacobian is 1-dimensional, therefore (infinitesimal) 
motion is possible only along one direction. While the explicit computation is somewhat tedious, there is a nice 
conceptual way to arrive at that conclusion. In fact, in this case the kinematic map happens to be the Gauss map
of the torus, and it is known that determinant of its Jacobian is precisely the Gauss curvature. Therefore, 
the singularities occur where the Gauss curvature is zero, i.e. along the top and the bottom parallels of the torus.

We are going to show that the above situation is not an exception. 
To this end let us examine singularities in spatial positioning  for
a serial chain consisting of revolute joints as in Figure \ref{fig: serial}. 
\begin{figure}[ht]
    \centering
    \includegraphics[scale=0.5]{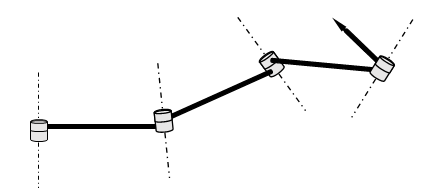}
    \caption{Serial chain with revolute joints.}
    \label{fig: serial}
\end{figure}
Several observations can be made. First, note that the joints can be always rotated so that all axes become parallel 
to some plane. Indeed, let us denote by $\vec x_i$ the vector product of the $i$-th and $(i+1)$-st joint axis 
in the chain, as in the Denavit-Hartenberg convention, and then apply the following procedure. The first two axes are 
parallel to a plane $P$, normal to $\vec x_1$, so we can rotate the 
second joint until $\vec x_2$ becomes aligned with $\vec x_1$. Then the first three axes are all parallel to $P$. 
Continue the rotations until all axes are parallel to $P$. 
Now observe that the system cannot move infinitesimally in the direction normal to the plane. 
In fact, the infinitesimal motion of the end-effector can be written as a vector sum
of infinitesimal motions of each joint. Clearly, if all axes are parallel to $P$, infinitesimal 
motion in the direction that is normal to $P$ is impossible. Conversely, if the vectors
$\vec x_i$ are not all aligned, then the infinitesimal moves span $\RR^3$, so the mechanism is not in singular
 position. Furthermore, whenever the serial chain is aligned 
in a singular configuration, we may rotate around the first and the last axis and clearly, still remain 
in a singular configuration. Therefore, the set of singular points of 
a serial chain is a union of two-dimensional tori. 

\begin{theorem}
A serial manipulator with revolute links is in singular position if, and only if, all axes are parallel to a plane.
Moreover, the set of singular points is 2-dimensional.
\end{theorem}

The 'if' part of the above theorem was known to Hollerbach \cite{Hollerbach}, and the general formulation is due to Gottlieb \cite{Gottlieb}). 
As a consequence, if the joint space is 3-dimensional, or more generally, if the robot arm is modular, 
so that three joints are used for positioning and the remaining joints take care of the orientation of the arm,  
then the singular space forms a separating subset (at least locally), which means
that in general we cannot avoid singularities while moving the robot arm between different positions 
in the working space. This problem can be avoided 
if the manipulator has more than three joints dedicated to the positioning of the arm. Even better, in redundant 
systems one can have that singular values always
contain regular points in their preimages, which at least in principle opens a possibility for regular motion planning. 

It is reasonable to ask whether one can completely eliminate singularities by constructing a robot arm with sufficiently
many revolute joints. Unfortunately, the answer is again negative, as Gottlieb \cite{Gottlieb} showed that every 
map from a torus to standard working spaces must have singularities.

\begin{theorem}
\label{thm:singularities}
Every smooth map $f\colon (S^1)^n\to \ws$ where $\ws=S^2, SO(3)$ or $SE(3)$ has singular points.
\end{theorem}
\begin{proof}
Assume that the Jacobian of $f$ is everywhere surjective.  
Then $f$ is a submersion, and therefore, by 
a classical theorem of Ehresmann \cite{Ehresmann}, it is a fibre bundle. It follows that the composition of $f$ 
with the universal covering map $p\colon \RR^n\to (S^1)^n$
yields a fibre bundle $f\circ p\colon \RR^n\to\ws$, whose fibre is a submanifold of $\RR^n$. 
But $\RR^n$ is contractible, so the fibre of $f\circ p$ is homotopy
equivalent to the loop space $\Omega\ws$. This contradicts the finite-dimensionality of the fibre, 
because the loop spaces $\Omega S^2$, $\Omega SO(3)$ and $\Omega SE(3)$ are known to have infinitely 
many non-trivial homology groups.
\end{proof}

Note that under the general assumptions of the theorem we do not have any information about the dimension of 
the singular set. 

\subsection{Redundant manipulators}

Although Theorem \ref{thm:singularities} implies that even redundant mechanisms have singularities, 
the extra room in the configuration spaces of redundant manipulators 
allows construction and extension of local inverse kinematic maps, so one can hope to achieve non-singular 
manipulation planning over large portions of the working space. In fact, there is a great deal of ongoing research 
that tries to exploit additional degrees of freedom in redundant manipulators. See for example the survey 
\cite{Tutorial}, or a more recent book chapter \cite{Kinematically Redundant Manipulators}
for possible approaches to inverse kinematic maps and manipulation, especially in  terms of differential equations 
and variational problems. In this subsection we will focus on some qualitative questions. To fix the ideas, let us 
assume that the mechanism has only revolute joints, so $\cs=(S^1)^n$, and  
consider either the pointing or orientation mechanisms, i.e. $\ws=S^2$ or $\ws=SO(3)$. Moreover, assume that 
$\dim \cs> \dim\ws$  so that the mechanism is redundant. 

A standard approach to robot manipulation is the following: given a robot's initial joint configuration 
$c\in\cs$ and end-effector's required final position $w\in\ws$, one first computes the initial position $f(c)\in \ws$, 
and finds a motion from $f(c)$ to $w$ represented by a path $\alpha\colon [0,1]\to\ws$  from $\alpha(0)=f(c)$ 
to $\alpha(1)=w$. Then the path $\alpha$ is lifted to $\cs$ starting from the initial point $c$, thus obtaining 
a path $\widetilde\alpha\colon[0,1]\to\cs$.  The lifting $\widetilde\alpha$ represents the motion of joints 
that steers the robot to the required position $w$, which is is reminiscent 
of the path-lifting problem in covering spaces or fibrations.  However, we know that the kinematic map is not 
a fibre bundle in general, and so the path lifting must avoid singular points. From the computational viewpoint the most
natural approach to path lifting is by solving differential equations but there are also other approaches.
We will follow \cite{Baker} and call any such lifting method a \emph{tracking algorithm}.

Clearly, every (smooth) inverse kinematic map determines a tracking algorithm but the converse is not true in general.
In fact, a tracking algorithm determined by inverse kinematics is always \emph{cyclic} in the sense that it lifts 
closed paths in $\ws$ to closed paths in $\ws$. Baker and Wampler \cite{Baker-Wampler} proved that 
a tracking algorithm is equivalent to one determined by inverse functions if, and only if the tracking is cyclic. 
Therefore Theorem \ref{sections} implies that, notwithstanding the available redundant degrees of freedom, one cannot
construct a cyclic tracking algorithm for pointing or orienting. This is not very surprising if we know that 
most tracking algorithms  rely on solutions of differential equations and are thus of a local nature. 
In particular the most widely used 
Jacobian method with additional constraints (cf. \emph{extended Jacobian method} in \cite{Tutorial} or
\emph{augmented Jacobian method} in \cite{Kinematically Redundant Manipulators}) yields tracking that is 
only \emph{locally cyclic}, in the sense that there is an open cover of $\ws$, such that  closed paths contained in 
elements of the cover are tracked by closed paths in $\cs$. 
However, this does not really help, because Baker and Wampler \cite{Baker-Wampler} (see also \cite[Theorem 2.3]{Baker})
proved that if there is a tracking algorithm defined on an entire  $\ws=S^2$ or $\ws=SO(3)$, then there 
are arbitrarily short 
closed paths in the working space that are tracked by open paths in $\cs$. Therefore 

\begin{theorem}
The extended Jacobian method (or any other locally cyclic method) cannot be used to construct a tracking
algorithm for pointing or orienting a mechanism with revolute joints.
\end{theorem}

Let us also mention that an analogous result can be proved for positioning mechanisms where the working space is a 2- 
or 3-dimensional disk around the base of the mechanism and whose radius is sufficiently big. See 
\cite[Theorem 2.4]{Baker} and the subsequent Corollary for details.

Our final result is again due to Gottlieb \cite{Gottlieb:RFB} and is related to the question of whether it is possible to restrict the
angles of the joints to stay away from the singular set $\mathrm{Sing}(f)$ of the Denavit-Hartenberg 
kinematic map $f\colon (S^1)^n\to SO(3)$. We obtain the following surprising restriction.

\begin{theorem}
\label{thm:closed}
Let $M$ be a closed smooth manifold. Then there does not exist a smooth map 
$s\colon M\to (\cs- \mathrm{Sing}(f))$ such that the map $f\circ s\colon M\to SO(3)$ is a submersion (i.e. non-singular).
\end{theorem}

The proof is based on a simple lemma which is also of independent interest.

\begin{lemma}
\label{lem:DH}
The Denavit-Hartenberg map $f\colon (S^1)^n\to SO(3)$ can be factored up to a homotopy as
$$\xymatrix{
(S^1)^n \ar[rr]^f \ar[dr]_m& &  SO(3)\\
& S^1 \ar[ur]_g}$$
where $m$ is the $n$-fold multiplication map in $S^1$ and $g$ is the generator of $\pi_1(SO(3))$.
\end{lemma}
\begin{proof}
First observe that a rotation of the $z$-axis in the Denavit-Hartenberg frame of any joint induces a homotopy between 
the resulting  forward kinematic maps. Therefore, we may deform the robot arm until all $z$-axes are parallel 
(so that the arm is effectively planar). Then the rotation angle of the end-effector is simply the sum of 
the rotations around each axis. As the sum of angles correspond to the multiplication in $S^1$, it follows that 
the Denavit-Hartenberg map factors as $f=g\circ m$ for some map $g\colon S^1\to SO(3)$. Clearly, $g$ must generate
$\pi_1(SO(3))=\ZZ_2$ because $f$ induces an epimorphism of fundamental groups.
\end{proof}

\begin{proof}(of Theorem \ref{thm:closed})
Assume that there exists $s\colon M\to (\cs- \mathrm{Sing}(f))$ such that $f\circ s\colon M\to SO(3)$ 
is a submersion. It is well-known that the projection of an orthogonal $3\times 3$-matrix to its last column determines a
map $p\colon SO(3)\to S^2$ that is also a submersion. 
Then the composition $$\xymatrix{M \ar[r]^-s &  \cs- \mathrm{Sing}(f)\ar[r]^-f & SO(3) \ar[r]^-p & S^2}$$
is a submersion, and hence a fibre bundle by Ehresmann's theorem \cite{Ehresmann}. It follows that
the fibre of $p\circ f\circ s$ is a closed submanifold of $M$. On the other side, $p\circ f\circ s$
is homotopic to the constant, because by Lemma \ref{lem:DH} $f$ factors through $S^1$ and $S^2$ is simply-connected. 
This leads to a contradiction, as the fibre of the constant map $M\to S^2$ is homotopy equivalent to 
$M\times \Omega S^2$, which cannot be homotopy equivalent to a closed manifold, because $\Omega S^2$ has 
infinite-dimensional homology.
\end{proof}

In particular, Theorem \ref{thm:closed} implies that even if we add constraints that restrict 
the configuration space  of the robotic device to some closed submanifold $\cs'$ of the set of non-singular 
configurations of the joints, the restriction  $f\colon \cs'\to SO(3)$ of the Denavit-Hartenberg kinematic map
still has singular points. Clearly, the new singularities  are not caused by the configuration of joints but 
are a consequence of the constraints that define $\cs'$.

\section{Overview of topological complexity}
\label{sec:TC overview}

The concept of topological complexity was introduced by M. Farber in \cite{Farber:TCMP} as a qualitative measure of 
the difficulty in constructing a robust motion plan for a robotic device. Roughly speaking, motion planning 
problem for some mechanical device requires to find a rule ('motion plan') that yields a continuous trajectory 
from any given initial position to a desired final position of the device. A motion plan is robust if small variations
in the input data results in small variations of the connecting trajectory.

Toward a mathematical formulation of the motion planning problem one
considers the space $\cs$ of 
all positions ('configurations') of the device, and the space $\pcs$ of all continuous paths $\alpha\colon I\to\cs$. Let 
$\pi\colon \pcs\to\cs\times \cs$ be the evaluation map given by 
$\pi(\alpha)=(\alpha(0),\alpha(1)).$
A \emph{motion plan} is a rule that to each pair of points $c,c'\in\cs$ assigns a path $\alpha(c,c')\in\pcs$ 
such that $\pi(\alpha(c,c'))=(c,c')$. 
For practical reason we usually require \emph{robust} motion plans, i.e. plans that are continuously dependent on $c$ 
and $c'$. Clearly, 
robust motion plans are precisely the continuous sections of $\pi$. Farber observed that a continuous global section of
$\pi$  exists if, and only if, $\cs$ is contractible. For non-contractible spaces one may consider partial continuous
sections, so he  defined 
the \emph{topological complexity} $\TC(\cs)$ as the minimal $n$ for which $\cs\times\cs$ can be covered by $n$ open sets
each admitting a continuous section to $\pi$. Note that other authors prefer the 'normalized' topological 
complexity (by one smaller that our $\TC$) as it sometimes leads to simpler formulas. 

We list some basic properties of the topological complexity:
\begin{enumerate}
\item It is a homotopy invariant, i.e. $\cs\simeq\cs'$ implies $\TC(\cs)=\TC(\cs')$.
\item There is a fundamental estimate 
$$\cat(\cs)\le\TC(\cs)\le\cat(\cs\times\cs),$$ 
where $\cat(\cs)$ denotes the (Lusternik-Schnirelmann) category of $\cs$ (see \cite{CLOT}).
\item Furthermore 
$$\TC(\cs)\ge \nil\big(\Ker\,\Delta^*\colon H^*(\cs\times\cs)\to H^*(\cs)\big).$$ 
Here $\Ker\,\Delta^*$ is the kernel of the homomorphism between
cohomology rings induced by the diagonal map $\Delta\colon\cs\to\cs\times\cs$, and $\nil(\Ker\,\Delta^*)$ is the minimal $n$
 such that every product of $n$ elements in
$\Ker\Delta^*$ is zero.
\end{enumerate}
There are many other results and explicit computations of topological complexity -- see \cite{Farber:ITR} for a fairly 
complete survey of the general theory.

\section{Complexity of a map}
\label{sec:Complexity of a map}

In \cite{Pav:CFKM} we extended Farber's approach to study the more general problem of robot manipulation. Robots are usually 
 manipulated by operating their joints in a way to achieve 
a desired pose of the robot or a part of it (usually called \emph{end-effector}), so we must take into account
the kinematic map which relates the internal 
joints states with the position and orientation of the end-effector. 
To model this situation, we take a map $f\colon \cs\to\ws$ and consider the
 projection map $\pi_f\colon \pcs\to \cs\times \ws$,
defined as 
$$\pi_f(\alpha):=(1\times f)(\pi(\alpha))= \big(\alpha(0),f(\alpha(1))\big).$$
Similarly to motion plans, a manipulation plan corresponds to a continuous sections of $\pi_f$, so it would be natural 
to define the topological complexity $\TC(f)$ as the  minimal $n$ such that $\cs\times\ws$ can be covered by $n$ sets,
each admitting a continuous section to $\pi_f$. This is analogous to the definition of $\TC(\cs)$, but there are two 
important issues that we must discuss before giving a precise description of $\TC(f)$.

In the definition of $\TC(\cs)$ Farber considers continuous sections whose domains are open subsets of $\cs\times\cs$. 
In most applications $\cs$ is a 
nice space (e.g. manifold, semi-algebraic set,...), and in fact Farber \cite{Farber:ITR} shows  that for such spaces 
alternative definitions of topological complexity based on closed, locally compact or ENR domains yield the same result. 
This was further generalized by Srinivasan \cite{Srinivasan} 
who proved that if $\cs$ is a metric ANR space, then every section over an arbitrary subset $Q\subset \cs$ can be extended 
to some open neighbourhood of $Q$. 
Therefore, for a very general class of spaces (including metric ANRs) one can define topological complexity 
by counting sections of the evaluation map $\pi\colon \pcs\to \cs\times \cs$ over arbitrary subsets of the base. 

Another important fact is that the map $\pi$ is a fibration, which
implies that  one can replace sections by homotopy sections in the definition of $\TC(\cs)$ and still get the same result.  
This relates topological complexity with the so called Schwarz genus \cite{Schwarz}, a well-established and extensively
studied concept in homotopy theory. The \emph{genus} $\g(h)$ of a map $h\colon X\to Y$ 
is the minimal $n$ such that $Y$ can be covered by $n$ open subsets, each admitting a continuous \emph{homotopy section} 
to $h$; the genus is infinite if  there is no such $n$. Therefore, we have 
$\TC(\cs)=\g(\pi)$, a result that puts topological complexity squarely within the realm of homotopy theory.

The situation is less favourable when it comes to the complexity of a map. Firstly, 
$\pi_f\colon\pcs\to\cs\times\ws$ is a fibration if, and only if, $f\colon\cs\to\ws$ is a fibration, and that is 
an assumption that we do not wish to make in view of our intended applications (cf. Theorem \ref{thm:singularities}). 
Every section is a homotopy section but not vice-versa, and in fact, the minimal number of homotopy sections for a given map 
can be strictly smaller than the number of sections. For example, the map $h\colon [0,3]\to [0,2]$ given by 
$$h(t):=\left\{\begin{array}{ll}
t & t\in [0,1]\\
1 & t\in [1,2]\\
t-1 & t\in [2,3]
\end{array}\right.$$
(see Figure \ref{fig: genus 1 sec 2})
\begin{figure}[ht]
    \centering
    \includegraphics[scale= 0.6]{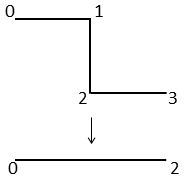}
    \caption{Projection with genus 1 and sectional number 2.}
    \label{fig: genus 1 sec 2}
\end{figure} 
admits a global homotopy section because its codomain is contractible, but clearly there does not exist 
a global section to $h$. Furthermore, the following example (which can be easily generalized) shows that the 
difference between the minimal number of sections and the minimal number of homotopy sections can be arbitrarily large.
\begin{figure}[ht]
    \centering
    \includegraphics[scale=0.6]{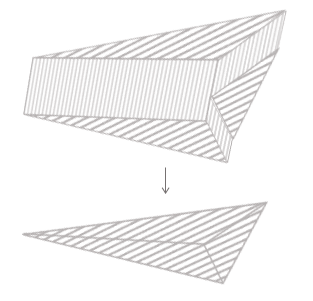}
    \caption{Projection with genus 1 and sectional number 3.}
    \label{fig: genus 1 sec 3}
\end{figure} 
Actually, many results on topological complexity depend heavily on the fact that the evaluation map 
$\pi\colon \pcs\to \cs\times \cs$ is a fibration, and so some direct generalizations of the results about $\TC(\cs)$ 
are harder to prove  while other are simply false.

The second difficulty is related to the type of subsets of $\cs\times\ws$ that are domains of sections to $\pi_f$.
While the spaces $\cs$ or $\ws$ are usually assumed to be nice (e.g. manifolds), the map $f$ can have singularities
which leads to the problem that we explain next. 

Given a subset $Q$ of $\cs\times\ws$ and a point $c\in\cs$ let $Q|_c$ be the subset of $\ws$ defined as 
$Q|_c:=\{w\in\ws\mid (c,w)\in Q\}$. Assume that $Q|_c$ is non-empty and 
that there is a partial section $\alpha\colon Q\to\pcs$ to $\pi_f$. Then the map 
$$\alpha_c\colon Q|_c\to \cs\ \ \ \ \ \alpha_c(w):=\alpha(c,w)(1)\; .$$
satisfies $f(\alpha_c(w))=w$, so $\alpha_c$ is a partial section to $f$. Furthermore 
$H\colon Q|_c\times I\to\cs$, given by $H(w,t):=\alpha(c,w)(1-t)$ deforms the image of the section 
$\alpha_c(Q|_c)$ in $\cs$ to the point $c$, while $f\circ H$ deforms $Q|_c$ in $\ws$ to the point $f(c)$. 
These observations have several important consequences.

Assume that $(c,w)$ is an interior point of some domain $Q\subset\cs\times\ws$ of a partial section to $\pi_f$. Then $w$ is 
an interior point of $Q|_c$ that admits a partial section to $f$. 
Therefore, if $f$ is not locally sectionable around $w$ (like  the previously considered map $h\colon [0,3]\to [0,2]$ around the point 1), then it is impossible to find an open cover of $\cs\times\ws$ that admits
partial sections to $\pi_f$. A similar argument shows that we cannot use closed domains for a reasonable definition
of the complexity of $f$. One way out would be to follow the approach by Srinivasan \cite{Srinivasan} and 
consider sections with arbitrary subsets as domains, but that causes problems elsewhere. After some balancing
we believe that the following choice is best suited for applications. 

Let $\cs$ and $\ws$ be path-connected spaces, and let $f\colon\cs\to\ws$ be a surjective map. 
The \emph{topological complexity} $\TC(f)$ of  $f$ is defined as the minimal $n$ for which there exists a filtration of 
$\cs\times\ws$ by closed sets
$$\emptyset=Q_0\subseteq Q_1\subseteq\ldots\subseteq Q_n=\cs\times\ws,$$
such that $\pi_f$ admits partial sections over $Q_i-Q_{i-1}$ for $i=1,2,\ldots,n$. By taking complements we obtain 
an equivalent definition based on filtrations of $\cs\times\ws$ by open sets. If $\ws$ is a metric ANR, then 
$\g(\pi_f)\le\TC(f)$, and the two coincide if $f$ is a fibration.

Suppose $\TC(f)=1$, i.e. there exists a section $\alpha\colon\cs\times\ws\to\pcs$ to $\pi_f$. Then $(\cs\times\ws)|_c=\ws$ for
 every $c\in\cs$, and by the above considerations, $f\colon\cs\to\ws$ admits 
a global section that embeds $\ws$ as a categorical subset of $\cs$. Even more, $\ws$ can be deformed to a point within
 $\ws$, so $\ws$ is contractible and $\TC(\ws)=\cat(\ws)=1$. 

To get a more general statement, let us say that a partial section $s\colon Q\to \cs$ to $f\colon \cs\to \ws$ is
 \emph{categorical} if $s(Q)$ can be deformed to a point within $\cs$. 
Then we define $\csec(f)$ to be the minimal $n$ so that there is a filtration
$$\emptyset=A_0\subseteq A_1\subseteq\ldots\subseteq A_n=\ws,$$
by closed subsets, such that $f$ admits a categorical section over $A_i-A_{i-1}$ for $i=1,\ldots,n$ 
(and $\csec(f)=\infty$, if no such $n$ exists).
For $\ws$ a metric ANR space we have $\csec(f)\ge\cat(\ws)$ because $A=f(s(A))$ is contractible in $\ws$ for every categorical
 section $s\colon A\to \cs$. If furthermore $f\colon \cs\to\ws$ is a fibration, then $\csec(f)=\cat(\ws)$. 

\begin{theorem}
Let $\ws$ be a metric ANR space and let $f\colon\cs\to\ws$ be any map. Then 
$$\cat(\ws)\le\csec(f)\le\TC(f)<\csec(f)+\cat(\cs)\:.$$
\end{theorem}
\begin{proof}
If $\TC(f)=n$, then we have shown before that there exists a cover of $\ws$ by $n$ sets, each admitting a categorical
 section to $f$, therefore $\TC(f)\ge\csec(f)$. 

As a preparation for the proof of the upper estimate, assume that $C\subseteq \cs$ admits a deformation 
$H\colon C\times I\to\cs$ to a point $c_0\in\cs$, and that $A\subseteq \ws$ admits a categorical section $s\colon A\to\cs$ 
with a deformation $K\colon s(A)\times I\to \cs$ to a point $c_1\in\cs$. In addition, let $\gamma\colon I\to\cs$ be a path
 from $c_0$ to $c_1$.
Then we may define a partial section $\alpha\colon C\times A\to \pcs$ by the formula
$$\alpha(c,w)(t):=\left\{\begin{array}{ll}
H(c,3t) & 0\le t\le 1/3\\
\gamma(3t-1) & 1/3\le t\le 2/3\\
K(s(w),2-3t)& 2/3\le t\le 1
\end{array}\right.$$
By assumption, there is a filtration of $\cs$ by closed sets 
$$\emptyset=C_0\subseteq C_1\subseteq\ldots\subseteq C_{\cat(\cs)}=\cs,$$
such that each difference $C_i-C_{i-1}$ deforms to a point in $\cs$,
and there is also a filtration of $\ws$ by closed sets
$$\emptyset=A_0\subseteq A_1\subseteq\ldots\subseteq A_{\csec(f)}=\ws,$$
such that each difference $A_i-A_{i-1}$ admits a categorical section to $f$. 
Then we can define closed sets 
$$Q_k=\bigcup_{i+j=k} C_i\times A_j$$
that form a filtration 
$$\emptyset=Q_1\subseteq Q_2\subseteq\cdots\subseteq Q_{\cat(\cs)+\csec(f)}=\cs\times\ws\;.$$ 
Each difference 
$$Q_k-Q_{k-1}=\bigcup_{i+j=k} (C_i-C_{i-1})\times (A_j-A_{j-1})$$
is a mutually separated union of sets that admit continuous partial sections, and so there exists a continuous 
partial section plan on each $Q_k-Q_{k-1}$. We conclude that $\TC(f)$ is less then or equal to $\cat(\cs)+\csec(f)-1$.
\end{proof}

Topological complexity of a map can be used to model several important situations in topological robotics. 
In the rest of this section we describe some typical examples:

\begin{example}
The identity map on $X$ is a fibration, so if $X$ is a metric ANR, then $\TC(X)=\TC(\Id_X)$. 
\end{example}

\begin{example}
In the motion planning of a device with several moving components one is often interested only in the position of a part of
 the system. This situation may be modelled by considering 
the projection $p\colon\cs\to\cs'$ of the configuration space of the entire system to the configuration space of the relevant
 part. Then $\TC(p)$ measures the complexity of robust motion planning
in $\cs$ but with the objective to arrive at a requested state in $\cs'$. A similar situation which often arises and can be
 modelled in this way is when the device can move and revolve in three 
dimensional space (so that 
its configuration space $\cs$ is a subspace of $\RR^3\times SO(3)$), but we are only interested in its final position (or
 orientation), so we consider the complexity of the projection $p\colon \cs\to\RR^3$  
(or $p\colon\cs\to SO(3)$). 
\end{example}

\begin{example}
Our main motivating example is the complexity of the forward kinematic map of a robot as introduced in \cite{Pav:CFKM}. A
 mechanical device consists of rigid parts connected by joints. As we explained in the first part of this paper 
(see also \cite[Section 1.3]{Waldron-Schmiedeler}), although 
there are many types of joints  only two of them are easily actuated by motors --
 \emph{revolute} joints (denoted as R) and \emph{prismatic} or \emph{sliding} 
joints (denoted as P). Revolute joints allow rotational movement so their states can be described by points on the circle $S^1$. Sliding joints allow linear movement with motion limits, so their states are 
described by points of the interval $I$. Other joints are usually passive and only restrict the motion of the device, so a typical configuration space  $\cs$ of a system with $m$ revolute
joints and $n$ sliding joints is a subspace of the product $(S^1)^m\times I^n$. Motion of the joints results in the spatial displacement of the device, in particular of its end-effector. 
The \emph{pose} of the end-effector is given by its spatial position and orientation, so the \emph{working space} $\ws$ of the device is a subspace of $\RR^3\times SO(3)$ (or a subspace of $\RR^2\times SO(2)$ if
the motion of the device is planar). In the following examples, for simplicity, we will disregard the orientation of the end-effector.

Given two revolute joints that are pinned together so that their axes of rotation are parallel, the configuration space 
is $\cs=S^1\times S^1$ and 
the working space is an annulus $\ws=S^1\times I$. The forward kinematic map $f\colon\cs\to\ws$ can be given explicitly 
in terms of polar coordinates. 
This configuration is depicted in Figure \ref{fig: two planar joints} with the mechanism and its working space overlapped,
and the complexity of the corresponding kinematic map is 3, see \cite[4.2]{Pav:CFKM}.

\begin{figure}[ht]
    \centering
    \includegraphics[scale=0.5]{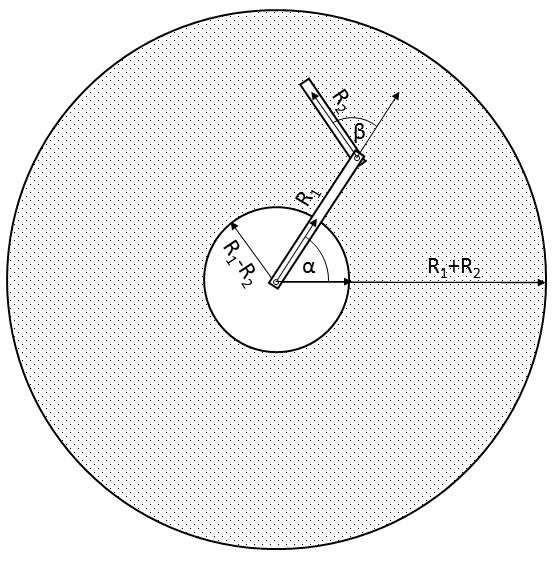}
    \caption{(RR) planar configuration. Position of the arm is completely described by the angles $\alpha$ and $\beta$, $\cs=S^1\times S^1$, $\ws=S^1\times[R_1-R_2,R_1+R_2]$.}
    \label{fig: two planar joints}
\end{figure}

If instead we pin the joints so that the axes of rotation are orthogonal then we obtain the so-called universal or Cardan
 joint. The configuration spaces is a product of circles
$\cs=S^1\times S^1$, but the working space is the two dimensional sphere and the forward kinematic map may be expressed by
 geographical coordinates (see Figure \ref{fig: two perpendicular joints}). By the computation in \cite[4.3]{Pav:CFKM} 
the complexity of the kinematic map for the universal joint is either 3 or 4 (we do not know the exact value).
\\
\begin{figure}[ht]
    \centering
    \includegraphics[scale=0.4]{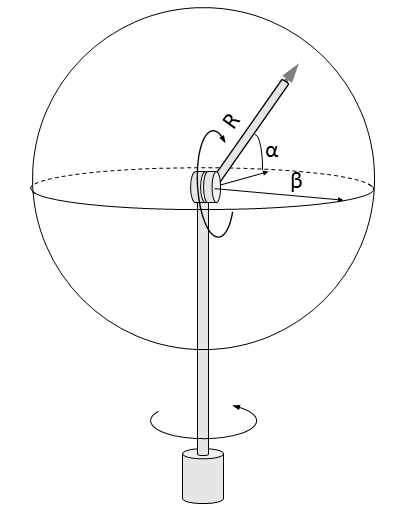}
    \caption{(RR) universal joint. $\cs=S^1\times S^1$, $\ws=S^2$.}
    \label{fig: two perpendicular joints}
\end{figure}

One of the most commonly used joint configurations in robotics is SCARA (Selective Compliant Assembly Robot Arm), which is based on the (RRP) configuration as in Figure \ref{fig: SCARA}, and is
sometimes complemented with a screw joint or even with a 3 degrees-of-freedom robot hand. The configuration space is $\cs=S^1\times S^1\times I$ and the working space is $\ws=S^1\times I\times I$. 
The forward kinematic map may be easily given in terms of cylindrical coordinates. Since the
kinematic map is the product of the kinematic map for the planar two-arm mechanism and the identity map
on the interval, its complexity is equal to 3.  
\begin{figure}[ht]
    \centering
    \includegraphics[scale=0.5]{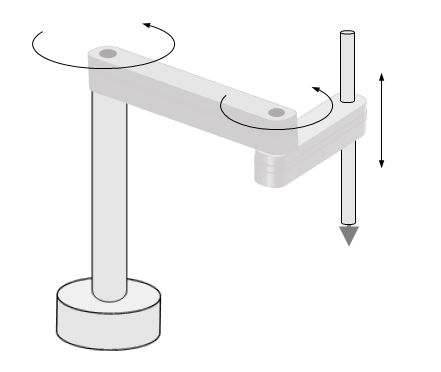}
    \caption{(RRP) SCARA design. $\cs=S^1\times S^1\times I$, $\ws=S^1\times I\times I$.}
    \label{fig: SCARA}
\end{figure}
\end{example}

The last computation can be easily generalized to products of arbitrary maps and we omit the proof since it follows the 
standard lines as in \cite[3.8]{Pav:FATC}.
\begin{proposition}
The product of maps $f\colon\cs\to\ws$ and $f'\colon \cs'\to\ws'$ satisfies the relation
$$\max\{\TC(f),\TC(f')\}\le\TC(f\times f')<\TC(f)+\TC(f')\;.$$
\end{proposition}


\begin{example} 
Robotic devices are normally employed to perform various functions and it often happens that different states of the device are functionally equivalent (say for grasping, welding, spraying or other purposes as in Figure \ref{fig: funct.eq.}). 
\begin{figure}[ht]
    \centering
    \includegraphics[scale=0.5]{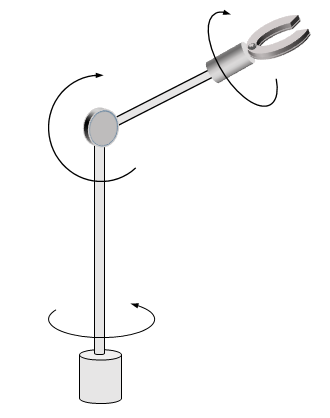}
    \caption{(RRR) design with functional equivalence for grasping.}
    \label{fig: funct.eq.}
\end{figure} 
Functional equivalence is often described by the action of some group $G$ on $\cs$ and there are several versions 
of equivariant topological complexity -- see \cite{Colman-Grant}, \cite{Lubawski-Marzantowicz},  \cite{Dranishnikov} 
or \cite{Błaszczyk-Kaluba}. Some of them require motion plans to be equivariant maps defined on invariant subsets of
$\cs\times\cs$, while other consider arbitrary paths that are allowed to 'jump' within the same orbit 
(see \cite[Section 2.2]{Błaszczyk-Kaluba} for an overview and comparison of different approaches). 
However, none of the mentioned papers give a convincing interpretation in terms of a motion planning problem for 
a mechanical system. 

We believe that the navigation planning for a device with a configuration space $\cs$ in which different 
configurations can have the same functionality should be modelled in terms of the complexity of the quotient map
 $q\colon\cs\to\cs/\sim$ associated to an equivalence relation on $\cs$ 
 (or $q\colon\cs\to\cs/G$ if the equivalence by a group action).
Then  $\TC(q)$ can be interpreted as  a measure of the difficulty in constructing a robust motion plan that steers 
a device from a given initial position to any of the final positions that have the required functionality. 
It would be interesting to relate this concept with the above mentioned versions of equivariant topological complexity.
\end{example}

\section{Instability of robot manipulation}
\label{sec:Instability of robot manipulation}

Let us again consider the robot manipulation problem determined by a forward kinematic map $f\colon\cs\to\ws$. 
A  \emph{manipulation algorithm} for the given device is a rule that to every 
initial datum $(c,w)\in\cs\times\ws$ assigns a path in $\cs$ starting at $c$ and ending at $c'\in f^{-1}(w)$. In other words,
 a manipulation algorithm is a (possibly discontinuous)
section of $\pi_f\colon\pcs\to \cs\times\ws$, patched from one or more robust manipulation plans. 

Let $\alpha_i\colon Q_i\to\pcs$ be a collection of robust manipulation plans such that the $Q_i$ cover $\cs\times\ws$. In general
 the domains $Q_i$ may overlap, 
so in order to define a manipulation algorithm for the device we must decide which manipulation plan to apply for a given
 input datum $(c.w)\in\cs\times\ws$. 
We can avoid this additional step by partitioning $\cs\times\ws$ into disjoint domains, e.g. by defining $Q'_1:=Q_1$,
 $Q'_2:=Q_2-Q_1$, $Q'_3:=Q_3-Q_2-Q_1,\ldots$,  
and restricting the respective manipulation plans  accordingly.


Since $\cs$ and $\ws$ are by assumption path-connected, if we partition $\cs\times\ws$ into several domains, then there exist
 arbitrarily close pairs of initial data $(c,w)$ and $(c',w')$ 
that belong to different domains.  This can cause instability of the robot device guided by such a manipulation
 algorithm in the sense that small perturbations of 
the input may lead to completely different behaviour  of the device. The problem is of particular significance when the input
 data is determined up to some approximation or rounding
because the instability may cause the algorithm to choose an inadequate manipulation plan. Also, a level of unpredictability
 significantly complicates coordination in a group of
robotic devices, because one device cannot infer the action of a collaborator just by knowing it's manipulation algorithm and
 by determining its position.  
Farber \cite{Farber:IRM} observed that for any motion planning algorithm the number of different choices that are available
 around certain
points  in $X$ increases with the topological complexity of $X$. He defined the \emph{order of instability} of a motion
 planning algorithm for $X$ at a point 
$(x,x')\in X\times X$ to be the number of motion plan domains that are intersected by every neighbourhood of $(x,x')$. Then he
 proved (\cite[Theorem 6.1]{Farber:IRM}) 
that for every motion planning algorithm on $X$ there is at least one point in $X\times X$ whose order of instability is at
 least $\TC(X)$.   

We are going to state and prove a similar result for the topological complexity of a forward kinematic map. Our proof is based
 on the approach used by Fox \cite{Fox} to tackle a similar question 
on Lusternik-Schnirelmann category.

\begin{theorem}
Let $f\colon\cs\to\ws$ be any map and let $\cs\times\ws=Q_1\sqcup\ldots\sqcup Q_n$ be a partition of $\cs\times\ws$ into disjoint subsets, each of them
admitting a partial section $\alpha_i\colon Q_i\to\pcs$ of $\pi_f$. Then there exists a point $(c,w)\in\cs\times\ws$ such that every neighbourhood of $(c,w)$ 
intersects at least $\TC(f)$ different domains $Q_i$. 
\end{theorem}
\begin{proof}
If every neighbourhood of $(c,w)$ intersects $Q_i$ then $(c,w)$ is in $\overline Q_i$, the closure of $Q_i$. Therefore, 
we must prove that there exist $\TC(f)$ different 
domains $Q_i$ such that their closures have non-empty intersection. To this end for each $k=1,2,\ldots, n$ we define 
$R_k$ as the set of points in
$\cs\times\ws$ that are contained in at least $k$ sets $\overline Q_i$. Each $R_k$ is a union of intersections of 
sets $\overline Q_i$, hence it is closed, and we obtain a filtration 
$$\cs\times\ws=R_1\supseteq R_2\supseteq\ldots R_m\supseteq\emptyset,$$
where $m$ is the biggest integer such that $R_m$ is non-empty.
For each $k=1,\ldots,m$ the difference $R_k-R_{k+1}$ consists of points that are contained in 
exactly $k$ sets $\overline Q_i$.
To construct a manipulation plan over $R_k-R_{k+1}$ let us define sets $S_I$ for every subset of indices 
$I\subseteq \{1,\ldots,n\}$ 
as the set of points that are contained in $\overline Q_i$ if $i\in I$ and are not contained in $\overline Q_i$ if 
$i\notin I$. It is easy to check that $R_k-R_{k+1}$ is the disjoint union of sets $S_I$ where $I$ ranges over
all $k$-element subsets of $\{1,\ldots,n\}$. Even more, if $I$ and $J$ are different but of the same
 cardinality, then the closure of $S_I$ does not intersect 
$S_J$ (i.e., $S_I$ and $S_J$ are \emph{mutually separated}). In fact, there is an index $i$ contained in $I$ but not 
in $J$, and clearly $\overline S_I\subseteq \overline Q_i$
while $S_J\cap \overline Q_i=\emptyset$. Since for $I$ of fixed cardinality $k$ the sets $S_I$ are mutually separated 
we can patch a continuous section $\beta_k\colon R_k-R_{k+1}\to \pcs$ to $\pi_f$ 
by choosing $i\in I$ for each $I$ and defining $\beta_k|_{S_I}:=\alpha_i|_{S_I}$. 

By definition of $\TC(f)$ we must have $m\ge\TC(f)$ so
 that $R_{\TC(f)}$ is non-empty and there exists a point in $\cs\times\ws$ that is
contained in at least $\TC(f)$ different sets $\overline Q_i$.
\end{proof}

The order of instability of a manipulation algorithm with $n$ robust manipulation plans clearly cannot exceed $n$, so there is always a cover of $\cs\times\ws$ by sets admitting section to $\pi_f$, whose order 
of instability is exactly $\TC(f)$. As a corollary we obtain the characterization of $\TC(f)$: it is the minimal $n$ for which every manipulation plan on $\cs\times\ws$ has order 
of instability at least $n$. 

In applications the robot manipulation problem is often solved numerically, using gradient flows or successive approximations. Again, one may identify domains of continuity as well as regions 
of instability. It would be an interesting project to compare different approaches with respect to their order of instability.


\begin{thebibliography}{99}
\bibitem{Baker}
D.~R.~Baker, \emph{Some Topological Problems in Robotics}, The Mathematical Intelligencer, 12 (1990) 66–-76.
\bibitem{Baker-Wampler}
D.~R.~Baker, C.~W.~Wampler, \emph{On the inverse kinematics of redundant manipulators}, International Journal of
Robotics Research, {\bf 7} (1988), 3--21.
\bibitem{Błaszczyk-Kaluba}
Z.~Błaszczyk, M.~Kaluba, \emph{Yet another approach to equivariant topological complexity}, arXiv: 1510.08724 (2015).
\bibitem{RS}
M. Brady (ed.), \emph{Robotics Science} (MIT Press, 1989).
\bibitem{Kinematically Redundant Manipulators}
S.~Chiaverini, G.~Oriolo, I.~D.~Walker, \emph{Kinematically Redundant Manipulators}, in B. Siciliano, O. Khatib
 (eds.), \emph{Springer Handbook of Robotics}, Chapter 1, (Springer, Berlin, 2008). 
\bibitem{Colman-Grant}
H.~Colman, M.~Grant, \emph{Equivariant topological complexity}, Algebr. Geom. Topol. {\bf 12} (2012), 2299–-2316.
\bibitem{CLOT}
O. Cornea, G. Lupton, J. Oprea, D. Tanr\'e, \emph{Lusternik-Schnirelmann category}, Mathematical Surveys and Monographs 
{\bf 103} (American Mathematical Society, Providence, RI, 2003).
\bibitem{D-H}
J. Denavit, R.S. Hartenberg, \emph{A kinematic notation for lower-pair mechanisms based on matrices}
Trans ASME J. Appl. Mech. {\bf 23} (1955), 215–-221.
\bibitem{Donelan}
P.~S.~Donelan, \emph{Singularities of robot manipulators}, in Singularity Theory, pp. 189–217, World Scientific, 
Hackensack NJ.
\bibitem{Dranishnikov}
A.~Dranishnikov, \emph{On topological complexity of twisted products}, Topology Appl. 179 (2015), 74–-80. 
\bibitem{Ehresmann}
C. Ehresmann, Les connexions infinit\'esimales dans un espace fibr\'e diff\'erentiable, in
 \emph{Colloque de Topologie}, Bruxelles (1950), 29--55.
\bibitem{Farber:TCMP}
M. Farber, Topological Complexity of Motion Planning, \emph{Discrete Comput Geom}, {\bf 29} (2003), 211--221.
\bibitem{Farber:IRM}
M. Farber, Instabilities of robot motion, \emph{Top. Appl.} {\bf 140} (2004), 245–-266
\bibitem{Farber:ITR}
M. Farber, \emph{Invitation to topological robotics}, (EMS Publishing House, Zurich, 2008).
\bibitem{Fox}
R.H. Fox, \emph{On the Lusternik-Schnirelmann category}, Ann. of Math. {\bf 42} (1941), 333--370. 
\bibitem{Gottlieb:IEEE}
D. Gottlieb, Robots and topology, \emph{Proc. 1986 IEEE International Conference on Robotics and Automation}, 
Vol 3., 1689--1691.
\bibitem{Gottlieb:RFB}
D.~Gottlieb, \emph{Robots and Fibre Bundles}, Bull. Soc. Math. Belgique {\bf 38} (1986), 219--223.
\bibitem{Gottlieb}
D. Gottlieb, Topology and the Robot Arm, \emph{Acta Applicandae Math.} {\bf 11} (1988), 117--121.
\bibitem{Hollerbach} 
J.M. Hollerbach, \emph{Optimal kinematic design for a seven degree of freedom manipulator}, 2nd International Symposium 
on Robotics Research, Kyoto, (Japan, 1984).
\bibitem{Lubawski-Marzantowicz}
W.~Lubawski, W.~Marzantowicz, \emph{Invariant topological complexity}, Bull. London Math. Soc. {\bf 47} (2014), 101–-117.
\bibitem{MIRM}
R.~M.~Murray, Z.~Li, S.~Shankar Sastry, \emph{A Mathematical Introduction to Robotic Manipulation}, (CRC Press, 1004)
\bibitem{Pav:FATC}
P.~Pave\v si\'c, \emph{Formal aspects of topological complexity}, in A.K.M. Libardi (ed.), \emph{Zbirnik prac´ 
Institutu matematiki NAN Ukraini} ISSN 1815-2910, T. 6, (2013), 56--66.
\bibitem{Pav:CFKM}
P.~Pave\v si\'c, \emph{Complexity of the Forward Kinematic Map}, to appear.
\bibitem{Schwarz}
A.S.~Schwarz, \emph{The genus of a fiber space}, Amer. Math. Soc. Transl. (2) {\bf 55} (1966), 49--140.
\bibitem{Tutorial}
B.~Siciliano, \emph{Kinematic Control of Redundant Manipulator: A Tutorial}, Journal of Intelligent and Robotic Systems 
{\bf 3} (1990), 201--212. 
\bibitem{Srinivasan}
T.~Srinivasan, \emph{The Lusternik-Schnirelmann category of metric spaces }, \emph{Topology Appl.} 167 (2014), 87--95.
\bibitem{Waldron-Schmiedeler}
K. Waldron, J. Schmiedeler, \emph{Kinematics}, in B. Siciliano, O. Khatib (eds.), 
\emph{Springer Handbook of Robotics}, Chapter 1, (Springer, Berlin, 2008). 
\end{thebibliography}
\end{document}